\documentclass{amsart}
\usepackage{amsmath, amscd, amssymb, amsthm}
\usepackage{bbm}
\usepackage{latexsym}
\usepackage{amsfonts}
\usepackage{graphicx}
\usepackage[all,cmtip]{xy}
\usepackage[colorlinks,linkcolor=blue,breaklinks=blue,urlcolor=blue,citecolor=blue,anchorcolor=blue,pagebackref]{hyperref}%
\usepackage{geometry}
\newtheorem{theorem}{Theorem}
\newtheorem{lemma}{Lemma}

\newtheorem*{definition}{Definition}
\newtheorem{conjecture}{Conjecture}

\renewcommand*\backref[1]{}
\renewcommand*\backrefalt[4]{ \ifcase #1 \or (cited on page #2) \else (cited on pages #2) \fi}

\newcommand{\be}{\begin{equation}}
\newcommand{\ee}{\end{equation}}
\newcommand{\bea}{\begin{eqnarray}}
\newcommand{\eea}{\end{eqnarray}}

\newcommand{\vs}{\vspace{0.5cm}}

\def\XXint#1#2#3{{\setbox0=\hbox{$#1{#2#3}{\int}$ }
\vcenter{\hbox{$#2#3$ }}\kern-.6\wd0}}

\begin{document}

\title[Streets-Tian Conjecture holds for 2-step solvmanifolds]{Streets-Tian Conjecture holds for 2-step solvmanifolds}

\author{Shuwen Chen}
\address{Shuwen Chen. School of Mathematical Sciences, Chongqing Normal University, Chongqing 401331, China}
\email{{3153017458@qq.com}}\thanks{Chen is supported by Chongqing graduate student research grant No.\,CYB240227. The corresponding author Zheng is partially supported by NSFC grants 12141101 and 12471039, by Chongqing grant cstc2021ycjh-bgzxm0139, by Chongqing Normal University grant 24XLB026, and is supported by the 111 Project D21024.}

\author{Fangyang Zheng}
\address{Fangyang Zheng. School of Mathematical Sciences, Chongqing Normal University, Chongqing 401331, China}
\email{20190045@cqnu.edu.cn; franciszheng@yahoo.com} \thanks{}


\begin{abstract}
A Hermitian-symplectic metric is a Hermitian metric whose K\"ahler form is given by the $(1,1)$-part of a closed $2$-form. Streets-Tian Conjecture states that a compact complex manifold admitting a Hermitian-symplectic metric must be K\"ahlerian (i.e., admitting a K\"ahler metric).  The conjecture is known to be true in dimension $2$ but is still open in dimensions $3$ or higher. In this article, we confirm the conjecture for all 2-step solvmanifolds, namely, compact quotients of 2-step solvable Lie groups by discrete subgroups. In the proofs, we adopted a method of using special {\em non-unitary} frames, which enabled us to squeeze out some hidden symmetries to make the proof go through. Hopefully the technique could be further applied. 
\end{abstract}

\maketitle

\tableofcontents

\section{Introduction and statement of results}\label{intro}

Streets and Tian introduced in \cite{ST} the notion of {\em Hermitian-symplectic metric,} namely a Hermitian metric $g$ on a compact complex manifold  $M^n$ whose K\"ahler form $\omega$ is the $(1,1)$-part of a closed $2$-form. That is,  there exists a global $(2,0)$-form $\alpha$ on $M^n$ so that $\Omega = \alpha + \omega + \overline{\alpha}$ is closed. It is equivalent to the existence of a symplectic form (i.e., non-degenerate closed $2$-form) $\Omega$ such that $\Omega (x, Jx)>0$ for any non-zero tangent vector $x$. So a Hermitian-symplectic structure is also called a {\em symplectic structure taming a complex structure $J$} (\cite{EFV}). 

Hermitian-symplectic metrics are natural and interesting. It is an obvious way to mix a Hermitian structure with a symplectic one. The definition immediately implies that any Hermitian-symplectic metric would satisfy $\partial \overline{\partial} \omega =0$, namely, it is {\em pluriclosed,}  a type of special Hermitian metrics that are extensively studied (see for instance \cite{FinoTomassini} and the references therein). Hermitian-symplectic metrics also play an important role in the Hermitian curvature flow theory developed by Streets and Tian \cite{Streets, ST, ST11, ST12}. For instance, Ye \cite{Ye} shows that the Hermitian-symplectic  property is preserved under the pluriclosed flow of Streets and Tian. Hermitian-symplectic metrics enjoy other nice properties, including that it is stable under small deformations (\cite[Ch.12]{OV}), and by \cite[Lemma 1]{YZZ} its presence guarantees the existence of a {\em strongly Gauduchon} metric, which means a Hermitian metric whose K\"ahler form $\omega$ satisfies $\partial \omega^{n-1} = \overline{\partial}\Phi$ where $\Phi$ is a global $(n,n-2)$-form on the $n$-dimensional manifold. An interesting and important open question in non-K\"ahler geometry is the following:

\begin{conjecture}[{\bf Streets-Tian \cite{ST}}]
If a compact complex manifold admits a Hermitian-symplectic metric, then it must admit a K\"ahler metric.
\end{conjecture}

The conjecture is known to be true in complex dimension $2$ (\cite{LiZhang, ST}), but open in higher dimensions in general. Complex dimension $2$ (that is, real dimension $4$) is rather special. See also \cite{Donaldson} for a related more general conjecture in real dimension $4$. In complex dimension $3$ or higher, the conjecture is only known in some special cases, including the following:
\begin{itemize}
\item twistor spaces: Verbitsky \cite{Verbitsky} showed that any non-K\"ahlerian twistor space does not admit any pluriclosed metric. 
\item Fujiki ${\mathcal C}$ class manifolds: namely compact complex manifolds bimeromorphic to compact K\"ahler manifolds. Chiose \cite{Chiose} proved that any non-K\"ahlerian Fujiki ${\mathcal C}$ class manifold does not admit any pluriclosed metric. 
\item the special non-K\"ahler Calabi-Yau threefolds of Fu-Li-Yau \cite{FuLiYau} via conifold transition: they showed that such threefolds do not admit any pluriclosed metric.
\item compact Chern flat manifolds: Di Scala-Lauret-Vezzoni \cite[Proposition 3.3]{DLV} showed that such spaces do not admit any pluriclosed metric. See also \cite[Cor 3.2]{FKV}.
\item Oeljeklaus-Toma manifolds: Fino-Kasuya-Vezzoni showed that \cite{FKV} showed that such manifolds do not admit any pluriclosed metric.
\item Vaisman manifolds: Angella-Otiman \cite{AOt} showed that any (non-K\"ahler) Vaisman manifold does not admit Hermitian-symplectic metrics.
\item non-balanced Bismut torsion parallel (BTP) manifolds: Guo and the second named author \cite{GuoZ2} showed that the conjecture holds.
\end{itemize}

Recall that a {\em Vaisman manifold} is a special type of locally conformally K\"ahler manifolds where the Lee form is parallel under the Levi-Civita connection. We refer the readers to the recent book \cite{OV} by Ornea and Verbitsky, which forms an encyclopedia for locally conformally K\"ahler and Vaisman manifolds.  A Hermitian metric is said to be {\em balanced} if its K\"ahler form $\omega$ satisfies $d(\omega^{n-1})=0$.  BTP manifolds are Hermitian manifolds whose Bismut connection (\cite{Bismut}) has parallel torsion. We refer the readers to \cite{ZhaoZ22, ZhaoZ24} for more discussions on this interesting type of special Hermitian manifolds. Andrada and Villacampa \cite{AndV} showed that  Vasiman manifolds are always non-balanced BTP. The latter also contains other examples such as {\em Bismut K\"ahler-like} manifolds (including all Bismut flat manifolds) and complex nilmanifolds that are BTP, when the dimension is $3$ or higher. See \cite{YZZ, ZZ-Crelle, ZZ-JGP} for more discussions on such metrics.

Next let us consider a large and interesting class of compact complex manifolds: {\em Lie-complex manifolds}. This means a compact complex manifold $M^n= G/\Gamma$ where $G$ is a Lie group and $\Gamma$ is a discrete subgroup, with the complex structure $J$ (when lifted onto $G$) being left-invariant. When $G$ is nilpotent or solvable, $M$ will be called a complex nilmanifold or complex solvmanifold, respectively. Regarding Streets-Tian Conjecture on Lie-complex manifolds, an important supporting evidence is the theorem of Enrietti-Fino-Vezzoni \cite{EFV}, which confirmed the conjecture for all complex nilmanifolds.  See also \cite{AN} for an important supplement to this.  For complex sovmanifolds, namely, when $G$ is solvable, the conjecture is confirmed in the following special situations (below let us denote by ${\mathfrak g}$ the Lie algebra of $G$):

\begin{itemize}
\item when ${\mathfrak g}$ is completely solvable: by Fino-Kasuya \cite{FK}.
\item when ${\mathfrak g}$ contains a $J$-invariant nilpotent complement: by Fino-Kasuya-Vezzoni in \cite{FKV}.
\item when $J$ is abelian: by Fino-Kasuya-Vezzoni \cite{FKV}.
\item when ${\mathfrak g}$ is almost abelian: by Fino-Kasuya-Vezzoni \cite{FKV} and Fino-Paradiso \cite{FP4}. See also \cite{GuoZ2} for an alternative proof. 
 \item when ${\mathfrak g}$ contains an abelian ideal ${\mathfrak a}$ of codimension $2$ with $J{\mathfrak a} = {\mathfrak a}$: by Guo-Zheng \cite{GuoZ2}.
 \item when ${\mathfrak g}$ contains an abelian ideal ${\mathfrak a}$ of codimension $2$ with $J{\mathfrak a} \neq {\mathfrak a}$: by Cao-Zheng \cite{CaoZ2}.
\end{itemize}

Recall that ${\mathfrak g}$ is {\em completely solvable} if $\forall\, x\in {\mathfrak g}$, $\mbox{ad}_x$ has only real eigenvalues. A subalgebra ${\mathfrak c}$ of a solvable Lie algebra ${\mathfrak g}$ is called an {\em nilpotent complement}, if  ${\mathfrak c}+{\mathfrak n}={\mathfrak g}$ (not necessarily direct sum) where ${\mathfrak n}$ is the nilradical of ${\mathfrak g}$. A complex structure $J$ on ${\mathfrak g}$ is said to be {\em abelian,} if $[Jx,Jy]=[x,y]$ holds $\forall\, x,y\in {\mathfrak g}$. In this case ${\mathfrak g}$ is necessarily $2$-step solvable \cite{ABD}.

A Lie algebra ${\mathfrak g}$ is {\em almost abelian,} if it contains an abelian ideal of codimension $1$. Note that such a Lie algebra is always $2$-step solvable. Almost abelian groups have been actively studied in recent years, see for instance \cite{AO, AL, Bock, CM, FP1, FP2, GuoZ, LW, Paradiso}. In the almost abelian case, \cite[Theorem 1.4]{FKV} confirmed the conjecture under the additional assumption that either $\dim {\mathfrak g}=6$ or ${\mathfrak g}$ is not of type (I), meaning that there exists $x\in {\mathfrak g}$ so that $\mbox{ad}_x$ has an eigenvalue with non-zero real part. In \cite{FP4} the additional assumption was removed. 

Lie algebras containing abelian ideals of codimension $2$ are natural generalizations to the almost abelian case, and they are always solvable of step at most $3$. The $J$-invariant case is much easier algebraically speaking, and it was done in \cite{GuoZ2}, while the  $J{\mathfrak a} \neq {\mathfrak a}$ was done later in \cite{CaoZ2}, utilizing our previous computations in \cite{CaoZ} which confirmed the Fino-Vezzoni Conjecture on such Lie algebras.    

The main purpose of this article is to confirm Streets-Tian Conjecture in the $2$-step solvable case, which will include the almost abelian or abelian $J$ case as special circumstances. Recall that a Lie algebra ${\mathfrak g}$ is said to be $2$-step solvable, if it is not abelian, but its commutator  ${\mathfrak g}':= [{\mathfrak g}, {\mathfrak g}]$ is abelian. In \cite{FSwann, FSwann2}, Freibert and Swann give a systematic treatment on the Hermitian geometry of $2$-step solvable Lie algebras. In particular, they shows that for such Hermitian Lie algebras of {\em pure type,} which means either $J{\mathfrak g}'\cap {\mathfrak g}'=0$,  or $J{\mathfrak g}'={\mathfrak g}'$, or $J{\mathfrak g}'+{\mathfrak g}'={\mathfrak g}$, then the Fino-Vezzoni Conjecture holds. The latter (\cite{FinoVezzoni15, FinoVezzoni}) states that any compact complex manifold admitting a balanced and a pluriclosed metric must be K\"ahlerian. Despite the apparent resemblance between Streets-Tian Conjecture and Fino-Vezzoni Conjecture, these two conjectures do not imply one another.

The following is the main result of this article, which confirms confirm Streets-Tian Conjecture for all Lie-complex manifolds that are $2$-step solvable:

\begin{theorem} \label{thm}
Let $M^n=G/\Gamma $ be a compact complex manifold which is the quotient of a Lie group by a discrete subgroup and the complex structure (when lifted onto $G$) is left-invariant, such that $G$ is $2$-step solvable. If $M$ admits a Hermitian-symplectic metric, then it must admit a (left-invariant) K\"ahler metric.
\end{theorem}

Note that having a compact quotient means that $G$ will have a bi-invariant measure (\cite{Milnor}). So by the averaging trick of Fino and Grantcharov \cite{FG, EFV}, if $M^n=G/\Gamma $ admits a  Hermitian-symplectic (or pluriclosed) metric, then it will admit a left-invariant Hermitian-symplectic (or pluriclosed) metric. So in particular, in the above theorem one may assume that there is an left-invariant Hermitian-symplectic metric, and the statement implies that Streets-Tian Conjecture holds for all compact quotients of $2$-step solvable groups. 

A technical point in our proof was the use of special  {\em non-unitary} frames. Although it makes the geometric computation clumsier but it does simplify the algebraic complexity considerably and also reveals a hidden symmetry which was crucial for the proof to go through. It is our  hope that this new approach could see further applications, in particular, when combined with the techniques developed by Freibert and Swann, it might help us to prove Fino-Vezzoni Conjecture for all $2$-step solvable groups, namely, dropping the  `pure types' assumption in the aforementioned theorem of Freibert and Swann.

\vspace{0.3cm}

\section{Hermitian geometry on Lie-complex manifolds}

In the past decades, the Hermitian geometry of Lie-complex manifolds have been extensively studied from various aspects by many people, including A. Gray, S. Salamon, L. Ugarte, A. Fino, L. Vezzoni, F. Podest\`a, D. Angella, A. Andrada, and others. There is a large amount of literature on this topic, and here we will just mention a small sample: \cite{AU}, \cite{CFGU},  \cite{FP3}, \cite{FinoTomassini09},  \cite{GiustiPodesta}, \cite{Salamon}, \cite{Ugarte}, \cite{WYZ}. For more general discussions on non-K\"ahler Hermitian geometry with a broader view, see for example \cite{AI}, \cite{AT}, \cite{Fu}, \cite{STW}, \cite{Tosatti} and the references therein.

In this section, we begin with the basics on Lie-complex manifolds. For our later purpose of verifying Streets-Tian  Conjecture on $2$-step solvable groups,  we need to deal with complex tangent frames which might not be unitary. This will inevitably make our formula clumsier, but the trade off is that it will significantly reduce some of the algebraic complexity that we will be facing later. 

Let $M=G/\Gamma$ be a Lie-complex manifold, namely, a compact quotient of a Lie group $G$ by a discrete subgroup $\Gamma \subseteq G$, where the complex structure $J$ is left-invariant. To verify Streets-Tian Conjecture for Lie-complex manifolds, first let us recall the neat averaging lemma by Fino-Grantcharov and Ugarte \cite{FG, Ugarte}, which allows one to reduce the consideration to invariant metrics. The following paragraphs are adopted from \cite{GuoZ2}, which we include here for readers convenience. 

Since the Lie group $G$ has compact quotient, it admits a bi-invariant measure by \cite{Milnor}. If $g$ is a Hermitian metric with K\"ahler form $\omega$, then by averaging $\omega$ over $M$, one gets an invariant Hermitian metric $\tilde{g}$ with K\"ahler form $\tilde{\omega}$. Clearly, if $g$ is pluriclosed, then so would be $\tilde{g}$. Similarly, if $g$ is a Hermitian-symplectic metric on $M$, then there would be $(2,0)$-form $\alpha$ on $M$ so that $\Omega = \alpha + \omega + \overline{\alpha}$ is a closed. By averaging $\Omega$, we get an invariant closed $2$-form $\tilde{\Omega}$. Clearly its $(1,1)$-part is necessarily positive, thus giving an invariant Hermitian-symplectic metric. If $g$ is balanced, on the other hand, then $\omega^{n-1}$ is closed. By averaging it over, we get an invariant closed $(n-1,n-1)$-form $\Psi$, which is clearly strictly positive. Since any strictly positive $(n-1,n-1)$-form on a complex $n$-manifold would equal to $\omega_0^{n-1}$ for a unique positive $(1,1)$-form $\omega_0$, we get an invariant balanced metric as well. We refer the readers to \cite[Lemma 3.2]{EFV}) and \cite{FG, Ugarte} for more details. The point is:

\vspace{0.2cm}

 {\em In order to verify Streets-Tian Conjecture or Fino-Vezzoni Conjecture for Lie-complex manifolds, it suffices to consider invariant metrics.} 

\vspace{0.2cm}

Let ${\mathfrak g}$ be the Lie algebra of $G$, and $J$ corresponds to a complex structure (which for convenience we will still denote by $J$) on  ${\mathfrak g}$, which means a linear isomorphism $J: {\mathfrak g} \rightarrow {\mathfrak g}$ satisfying $J^2=-I$ and the integrability condition:
\begin{equation} \label{integrability}
[x,y] - [Jx,Jy] + J[Jx,y] + J[x,Jy] =0, \ \ \ \ \ \ \forall \ x,y \in {\mathfrak g}. 
\end{equation}
Similarly, a left-invariant metric $g$ on $G$ compatible with $J$ corresponds to an inner product $g=\langle , \rangle$ on ${\mathfrak g}$ such that $\langle Jx,Jy\rangle = \langle x,y\rangle$ for any $x,y\in {\mathfrak g }$.

As in \cite{GuoZ2}, let ${\mathfrak g}^{\mathbb C}$ be the complexification of ${\mathfrak g}$, and write ${\mathfrak g}^{1,0}= \{ x-\sqrt{-1}Jx \mid x \in {\mathfrak g}\} \subseteq {\mathfrak g}^{\mathbb C}$. The integrability condition (\ref{integrability}) means that ${\mathfrak g}^{1,0}$ is a complex Lie subalgebra of ${\mathfrak g}^{\mathbb C}$. Extend $g=\langle , \rangle $ bi-linearly over ${\mathbb C}$, and let $e=\{ e_1, \ldots , e_n\}$ be a basis of ${\mathfrak g}^{1,0}$, which will be called a {\em frame} of $({\mathfrak g},J)$ or $({\mathfrak g}, J,g)$.  Denote by $\varphi$ the coframe dual to $e$, namely, a basis of the dual vector space $({\mathfrak g}^{1,0})^{\ast}$ such that $\varphi_i(e_j)=\delta_{ij}$, $\forall$ $1\leq i,j\leq n$. We will write
\begin{equation} \label{CandD}
C^j_{ik} = \varphi_j( [e_i,e_k] ), \ \ \ \ \ \  D^j_{ik} = \overline{\varphi}_i( [\overline{e}_j, e_k] )
\end{equation}
for the structure constants, which is  equivalent to
\begin{equation} \label{CandD2}
[e_i,e_j] = \sum_k C^k_{ij}e_k, \ \ \ \ \ [e_i, \overline{e}_j] = \sum_k \big( \overline{D^i_{kj}} e_k - D^j_{ki} \overline{e}_k \big) .
\end{equation}
In dual terms, the above is also equivalent to the (first) structure equation:
\begin{equation} \label{eq:structure}
d\varphi_i = -\frac{1}{2} \sum_{j,k} C^i_{jk} \,\varphi_j\wedge \varphi_k - \sum_{j,k} \overline{D^j_{ik}} \,\varphi_j \wedge \overline{\varphi}_k, \ \ \ \ \ \ \forall \  1\leq i\leq n.
\end{equation}
Differentiate the above, we get the  first Bianchi identity, which is equivalent to the Jacobi identity in this case:
\begin{equation} \label{Bianchi}
\left\{  \begin{split}  \sum_r \big( C^r_{ij}C^{\ell}_{rk} + C^r_{jk}C^{\ell}_{ri} + C^r_{ki}C^{\ell}_{rj} \big) \ = \ 0,  \hspace{3.2cm}\\
 \sum_r \big( C^r_{ik}D^{\ell}_{jr} + D^r_{ji}D^{\ell}_{rk} - D^r_{jk}D^{\ell}_{ri} \big) \ = \ 0, \hspace{3cm} \\
 \sum_r \big( C^r_{ik}\overline{D^r_{j \ell}}  - C^j_{rk}\overline{D^i_{r \ell}} + C^j_{ri}\overline{D^k_{r \ell}} -  D^{\ell}_{ri}\overline{D^k_{j r}} +  D^{\ell}_{rk}\overline{D^i_{jr}}  \big) \ = \ 0,  \end{split} \right.  
\end{equation}
for any $1\leq i,j,k,\ell\leq n$. When $G$ has a compact quotient, it (or equivalently, its Lie algebra ${\mathfrak g}$) is necessarily unimodular, that is,  $\mbox{tr}(ad_x)=0$ for any $x\in {\mathfrak g}$. In terms of structure constants, we have
\begin{equation} \label{unimo}
{\mathfrak g} \ \, \mbox{is unimodular}  \ \ \Longleftrightarrow  \ \ \sum_r \big( C^r_{ri} + D^r_{ri}\big) =0 , \, \ \forall \ i.
\end{equation}
Note that so far we did not assume $e$ to be unitary. Let  $g=(g_{i\bar{j}})=(\langle e_i, \overline{e}_j\rangle )$ be the matrix of metric under the frame $e$. Denote by  $\nabla$ the Chern connection of $g$, and by $T$, $R$ the torsion and curvature tensor of $\nabla$. Then  $T( e_i, \overline{e}_j)=0$. Write $T(e_i,e_k)  = \sum_j T^j_{ik}e_j$. We have the following:

\begin{lemma}  \label{lemma1}
Given a Hermitian Lie algebra $({\mathfrak g}, J,g)$, let $e$ be a frame with dual coframe $\varphi$, and structure constants $C$ and $D$ be given by (\ref{CandD}). Then the Chern torsion components $T^j_{ik}$ are given by
\begin{equation} \label{torsion}
 T^j_{ik}= - C^j_{ik} -  \sum_{\alpha , \beta} D^{\beta}_{\alpha k} \,g^{\bar{\beta}j} g_{i \bar{\alpha}} + \sum_{\alpha , \beta} D^{\beta}_{\alpha i} \,g^{\bar{\beta}j} g_{k \bar{\alpha}} .
\end{equation}
\end{lemma}

\begin{proof}
If the frame $e$ is unitary, then it is well known (see for example \cite{VYZ}) that
$$ \nabla_{e_k}e_i = \sum_j D^j_{ik}e_j, \ \ \ \ \ \forall \ 1\leq i,k\leq n. $$
Therefore by definition we would have
$$ T^j_{ik}= \varphi_j (T(e_i,e_k))=\varphi_j (\nabla_{e_i}e_k - \nabla_{e_k}e_i - [e_i,e_k])= D^j_{ki}-D^j_{ik}-C^j_{ik}. $$
That is, the formula (\ref{torsion}) holds as $g=I$ in this case. When $e$ is not necessarily unitary, let us choose a non-degenerate matrix $A$ so that the new frame $\tilde{e}$ determined by $\tilde{e}_i=\sum_j (A^{-1})_{ij}e_j$ is unitary. Then $g=AA^{\ast}$, and $\tilde{\varphi} =\,^t\!A\varphi$. We have
$$ \tilde{D}^b_{ac} = \overline{\tilde{\varphi}}_a ( [\overline{\tilde{e}}_b, \tilde{e}_c]) = A^{\ast}_{a\alpha} (\overline{A}^{-1})_{b\beta} (A^{-1})_{c\gamma} \overline{\varphi}_{\alpha} ([\overline{e}_{\beta} , e_{\gamma}] ) =   A^{\ast}_{a\alpha} (\overline{A}^{-1})_{b\beta} (A^{-1})_{c\gamma} D^{\beta}_{\alpha \gamma }.  $$
Therefore,
\begin{eqnarray*}
 \varphi_j(\nabla_{e_k}e_i ) & = & (^t\!A^{-1})_{jb} A_{ia}A_{kc} \,\tilde{\varphi}_b(\nabla_{\tilde{e}_c} \tilde{e}_a ) \ \,= \  \,(^t\!A^{-1})_{jb} A_{ia}A_{kc} \tilde{D}^b_{ac} \\
 & = & (^t\!A^{-1})_{jb} A_{ia}A_{kc}  A^{\ast}_{a\alpha} (\overline{A}^{-1})_{b\beta} (A^{-1})_{c\gamma} \,D^{\beta}_{\alpha \gamma } \\
 & = & (A^{-1\ast})_{\beta b}(A^{-1})_{bj} A_{ia}  A^{\ast}_{a\alpha} A_{kc}(A^{-1})_{c\gamma} \,D^{\beta}_{\alpha \gamma } \\
 & = & (A^{-1\ast}A^{-1})_{\beta j} (A  A^{\ast})_{i\alpha} \,D^{\beta}_{\alpha k } \ \, = \ \, g^{\bar{\beta}j} g_{i\bar{\alpha}} \,D^{\beta}_{\alpha k}.
\end{eqnarray*}
Thus we obtain
\begin{eqnarray*}
 T^j_{ik} & = &  \varphi_j (\nabla_{e_i}e_k ) - \varphi_j(\nabla_{e_k}e_i ) - \varphi_j ([ e_i , e_k ]) \\
 & = & g^{\bar{\beta}j} g_{k\bar{\alpha}} \,D^{\beta}_{\alpha i} - g^{\bar{\beta}j} g_{i\bar{\alpha}} \,D^{\beta}_{\alpha k} - C^j_{ik}.
\end{eqnarray*}
This completes the proof of the lemma.
\end{proof}

To verify Streets-Tian Conjecture for  Lie-complex manifolds, we need the following characterization for Hermitian-symplectic metrics on such manifolds, which is a slight modification of \cite[Lemma 3]{GuoZ2} as here we dropped the assumption that the frame is unitary. We include the proof here for readers' convenience.
\begin{lemma} [\cite{GuoZ2}]    \label{lemmaGuo}
Let $({\mathfrak g}, J, g)$ be a Lie algebra equipped with a Hermitian structure. The metric $g$ is Hermitian-symplectic if and only if for any given frame $e$, there exists a skew-symmetric matrix $S$ such that
\begin{equation} \label{eq:HS}
\left\{ \begin{split} 
 \sum_r \big( S_{ri}C^r_{jk} + S_{rj} C^r_{ki}  + S_{rk} C^r_{ij} \big) =0,  \hspace{1.1cm}\\
 \sum_r \big( S_{rk}\overline{D^i_{rj}} - S_{ri} \overline{D^k_{rj}} \big) = - \frac{\sqrt{-1}}{2}\sum_r  T^r_{ik} g_{r\bar{j}},      \end{split} \right. \ \ \ \ \ \ \ \ \ \ \ \ \ \forall\ 1\leq i,j,k\leq n. 
 \end{equation}
\end{lemma}

\begin{proof}
Let $\varphi$ be the coframe dual to $e$. By definition, $g$ will be Hermitian-symplectic if and only if there exists a $(2,0)$-form $\alpha$ so that $\Omega = \alpha + \omega + \overline{\alpha}$ is closed. Here $\omega$ is the K\"ahler form of $g$, defined by
$$ \omega = \sqrt{-1}\sum_{i,j} g_{i\bar{j}} \, \varphi_i \wedge \overline{\varphi}_j  = \sqrt{-1} \,^t\!\varphi \,g \,\overline{\varphi}, \ \ \ \ \ \ \ \ g = (g_{i\bar{j}}) , \ \ g_{i\bar{j}} = \langle e_i, \overline{e}_j\rangle . $$
 This means that $\partial \alpha =0$ and $\overline{\partial}\alpha = - \partial \omega$. Write $\alpha = \sum_{i,k} S_{ik}\varphi_i\wedge \varphi_k$, where $S$ is a skew-symmetric matrix. Then we have
\begin{eqnarray*}
 \partial \alpha & = &  2\sum_{r,k} S_{rk} \partial \varphi_r \wedge \varphi_k \ = \ - \sum_{r,i,j,k} S_{rk} C^r_{ij} \, \varphi_i\wedge \varphi_j \wedge \varphi_k \\
 & = & -\frac{1}{3} \sum_{i,j,k} \{ \sum_r \big(  S_{rk} C^r_{ij} + S_{rj} C^r_{ki} + S_{ri} C^r_{jk} \big)\}  \,\varphi_i\wedge \varphi_j \wedge \varphi_k. 
\end{eqnarray*}
So by $\partial \alpha =0$ we get the first identity in (\ref{eq:HS}). Similarly, we have
\begin{eqnarray*}
 \overline{\partial} \alpha & = &  2\sum_{r,k} S_{rk} \overline{\partial}\varphi_r \wedge \varphi_k \ = \ - 2 \sum_{r,i,j,k} S_{rk} \overline{D^i_{rj}} \, \varphi_i\wedge \overline{\varphi}_j \wedge \varphi_k \\
 & = & - \sum_{i,j,k} \{ \sum_r \big( S_{rk} \overline{D^i_{rj}} - S_{ri} \overline{D^k_{rj}} \big)\}  \, \varphi_i\wedge \overline{\varphi}_j \wedge \varphi_k. 
\end{eqnarray*}
On the other hand,
$$ \partial \omega = \sqrt{-1} \,^t\!\tau g \,\overline{\varphi} = \frac{\sqrt{-1}}{2} \sum_{i,j,k,r} T^r_{ik} g_{r\bar{j}}\,\varphi_i\wedge \varphi_k \wedge \overline{\varphi}_j , $$
so by comparing the above two lines we get the second identity in (\ref{eq:HS}). This completes the proof of Lemma \ref{lemmaGuo}.
\end{proof}

For our later proofs, we will need the following property of Hermitian-symplectic manifolds from \cite[Lemma 2]{CaoZ2}:

\begin{lemma} [\cite{CaoZ2}] \label{lemmaCao}
Let $(M^n,g)$ be a compact Hermitian manifold. If $g$ is Hermitian-symplectic, then any non-negative $(p,p)$-form $\Phi$ which is $d$-exact  must be trivial.
\end{lemma}

\vspace{0.3cm}

\section{$2$-step solvable Lie algebras}

In \cite{FSwann, FSwann2}, Freibert and Swann systematically studied the Hermitian geometry of $2$-step solvable Lie algebras. Among other things, they give  characterizations to balanced and pluriclosed metrics on such algebras, and confirmed Fino-Vezzoni Conjecture for all such Lie algebras that are of the pure types. Let ${\mathfrak g}$ be a Lie algebra and $J$ a complex structure on ${\mathfrak g}$. Let ${\mathfrak g}' = [{\mathfrak g}, {\mathfrak g}]$ be the commutator. Then ${\mathfrak g}$ is $2$-step solvable if ${\mathfrak g}$ is not abelian but ${\mathfrak g}' $ is abelian.  Freibert and Swann wrote 
$$ {\mathfrak g} = {\mathfrak g}'_J \oplus (V\oplus JV) \oplus W, \ \ \ \ \ \ {\mathfrak g}' = {\mathfrak g}'_J \oplus V, $$
where ${\mathfrak g}'_J=J{\mathfrak g}' \cap {\mathfrak g}'$, $V$ is a complement subspace of  ${\mathfrak g}'_J$ in  ${\mathfrak g}'$, and $W$ is a $J$-invariant complement subspace of $J{\mathfrak g}'+{\mathfrak g}'$ in ${\mathfrak g}$. They introduced the following terminologies for ${\mathfrak g}$:
\begin{itemize} 
\item Pure type I: if  $J{\mathfrak g}'\cap {\mathfrak g}'=0$,
\item Pure type II: if $J{\mathfrak g}'={\mathfrak g}'$,
\item Pure type III: if $J{\mathfrak g}'+{\mathfrak g}'={\mathfrak g}$.
\end{itemize}
That is, in their decomposition of ${\mathfrak g}$ into three summands, if at least one of the summand is zero, then the Lie algebra is said to be of pure type. Of course when $({\mathfrak g}, J)$ is equipped with a metric $g$, then one can choose $V$ to be $({\mathfrak g}'_J)^{\perp} \cap {\mathfrak g}'$, and $W = ({\mathfrak g}' + J {\mathfrak g}')^{\perp}$ in ${\mathfrak g}$, so the three summands are perpendicular to each other. Note however within the middle summand, the subspaces $V$ and $JV$ may not be perpendicular to each other.   

\begin{definition}[{\bf admissible frames}] 
Let $({\mathfrak g}, J,g)$ be a $2$-step solvable Lie algebra equipped with a Hermitian structure. Write $W=({\mathfrak g}' + J {\mathfrak g}')^{\perp}$ and $V=({\mathfrak g}'_J)^{\perp} \cap {\mathfrak g}'$. Then a frame $e$ is said to be {\em admissible} if $\{ e_1, \ldots , e_r\} $ is a unitary basis of $({\mathfrak g}'_J)^{1,0}$, and $\{ e_{s+1}, \ldots , e_n\}$ is a unitary basis of $W^{1,0}$, while $V$ is spanned by $ e_{\alpha} + \overline{e}_{\alpha}$ for $r\!+\!1\leq \alpha \leq s$. 
\end{definition}

Here we denoted by $2r$ the (real) dimension of ${\mathfrak g}'_J$ and by $2(n-s)$ the (real) dimension of $W$. We have $0\leq r\leq s\leq n$. From now on, we will always follow the index range convention:
\begin{equation} \label{convention}
1\leq i,j,k,\ldots \leq r;  \  \ \ r\!+\!1\leq \alpha , \beta , \gamma , \ldots  \leq s;   \ \ \ s\!+\!1\leq a, b, c, \ldots , \leq  n.
\end{equation}
Under admissible frames, the structure constants for $2$-step solvable Lie algebras satisfy the following restrictions:

\begin{lemma} \label{restriction1}
Let  $({\mathfrak g}, J,g)$ be a $2$-step solvable Lie algebra equipped with a Hermitian structure, and $e$ be an admissible frame. Then 
$$ \left\{ \begin{split}  C^{\ast}_{ij}=C^{\alpha}_{\ast \ast} = C^{a}_{\ast \ast}= D^{\ast}_{a\ast} = D^i_{\ast j} = D^{\alpha}_{\ast j}=D^i_{\alpha \beta} = 0, \\
 C^j_{i\alpha}=-\overline{D^i_{j\alpha}}, \ \ \ \  C^{\ast}_{\alpha \beta} = \overline{ D^{\beta}_{\ast \alpha} }  - \overline{ D^{\alpha}_{\ast \beta} },  \hspace{2.0cm} \\  \overline{D^x_{\alpha y}} = - D^y_{\alpha x}, \ \ \ \ \forall \ 1\leq x,y\leq n.  \hspace{2.55cm}
 \end{split} \right.
 $$
Here and from now on the indices $i,j,\alpha, \beta$ all follow the index range convention (\ref{convention}), while $\ast$ stands for an arbitrary integer between $1$ and $n$. 
\end{lemma}

\begin{proof}
By definition, ${\mathfrak g}'$ is spanned by $e_i+\overline{e}_i$, $\sqrt{-1}(e_i-\overline{e}_i)$ for all $1\leq i\leq r$ and $e_{\alpha}+\overline{e}_{\alpha}$ for all $r\!+\!1\leq \alpha \leq s$. Since ${\mathfrak g}'$ is abelian, we have
$$ [e_i, e_j]=[e_i, \overline{e}_j] = [e_i, e_{\alpha}+\overline{e}_{\alpha}] = [e_{\alpha}+\overline{e}_{\alpha}, e_{\beta}+\overline{e}_{\beta}]=0. $$
So by (\ref{CandD2}) we get $C^{\ast}_{ij}=D^i_{\ast j}=0$, $C^{\ast}_{i\alpha} + \overline{D^i_{\ast \alpha}} =0$, $D^{\alpha}_{\ast i}=0$, and $C^{\ast}_{\alpha \beta} = \overline{ D^{\beta}_{\ast \alpha} }  - \overline{ D^{\alpha}_{\ast \beta} }$. Now suppose that $X=\sum_{t=1}^nX_te_t$ is any type $(1,0)$ vector. Then clearly $X+\overline{X}\in {\mathfrak g}'$ if and only if $X_a=0$ for each $a$ and $X_{\alpha}\in {\mathbb R}$ for each $\alpha$. For any $1\leq x,y\leq n$, let
$$ X=[e_x, e_y] = \sum_{t=1}^n C^{t}_{xy} e_t \ \ \ \ \mbox{or} \ \ X=[\sqrt{-1}e_x, e_y] = \sum_{t=1}^n (\sqrt{-1}C^{t}_{xy}) e_t, $$
Since in both cases $X+\overline{X}$ belongs to ${\mathfrak g}'$, we get $C^a_{xy}=0$, $C^{\alpha}_{xy}=0$. Similarly, by considering
$$ [e_x, \overline{e}_y] + \overline{[e_x, \overline{e}_y] } = \sum_t \{ (\overline{D^x_{t y}} - \overline{D^y_{t x}})e_t + (D^x_{t y}- D^y_{t x})\overline{e}_t \} $$
and 
$$ [\sqrt{-1}e_x, \overline{e}_y] + \overline{[\sqrt{-1}e_x, \overline{e}_y] } = \sum_t \{ \sqrt{-1}(\overline{D^x_{t y}} +\overline{D^y_{t x}})e_t - \sqrt{-1} (D^x_{t y}+ D^y_{t x})\overline{e}_t \}, $$
both are in ${\mathfrak g}'$, so for $t=a$ we get $D^{\ast}_{a\ast}=0$ for any $a$, while for $t=\alpha$ we know that 
$$ D^x_{\alpha y}- D^y_{\alpha x} \in {\mathbb R}, \ \ \ \sqrt{-1}(D^x_{\alpha y}+ D^y_{\alpha x})\in {\mathbb R} , \ \ \ \ \forall \ r\!+\!1\leq \alpha \leq s. $$
This means that $\overline{D^x_{\alpha y}}= -D^y_{\alpha x}$. Note that by $C^{\ast}_{i\alpha} + \overline{D^i_{\ast \alpha}} =0$ and $C^{\alpha}_{\ast\ast}=0$, we get $D^i_{\beta \alpha}=0$. Similarly, by $C^{\ast}_{\alpha \beta} = \overline{ D^{\beta}_{\ast \alpha} }  - \overline{ D^{\alpha}_{\ast \beta} }$, we know that $D^{\alpha}_{\gamma \beta}= D^{\beta}_{\gamma \alpha}$ is pure imaginary. This completes the proof of the lemma.
\end{proof}

\begin{lemma} \label{restriction2}
Let $(M^n,g)$ be a compact Hermitian-symplectic manifold. Suppose that $M=G/\Gamma$ is the quotient of a $2$-step solvable Lie group by a discrete subgroup and the complex structure $J$ is left-invariant. Let $({\mathfrak g}, J,g)$ be the corresponding Hermitian Lie algebra. Then under any admissible frame $e$, the structure constants satisfy
$$ D^{\ast}_{\alpha \ast} =0, \ \ \ \ \forall \ r\!+\!1 \leq \alpha \leq s. $$
\end{lemma}

\begin{proof}
Let $\varphi$ be the coframe dual to $e$. By (\ref{eq:structure}) and Lemma \ref{restriction1}, we have $d\varphi_{a}=0$ for each $a$ and
$$ d\varphi_{\alpha}= - \sum_{x,y=1}^n \overline{D^x_{\alpha y} } \,\varphi_x \wedge \overline{\varphi}_y .$$
If $d\varphi_{\alpha}\neq 0$, then let us fix this $\alpha$ and consider the $n\times n$ matrix $A=(D^x_{\alpha y})_{1\leq x,y\leq n}$. Denote by $k$ the rank of $A$. It is skew-Hermitian, so there exists a unitary matrix $U$ and real numbers $\{ \lambda_1, \ldots , \lambda_n \}$ so that 
$$ A_{xy} = \sqrt{-1} \sum_{t=1}^n  \lambda _t U_{xt} \overline{U_{yt}}. $$
Without loss of generality, we may assume that  $\lambda_{k+1}= \cdots = \lambda_n=0$ while $\lambda_1 \cdots   \lambda_k \neq 0$.  Write
$$ \psi_t = \sum_{x=1}^n U_{xt} \varphi_x. $$
Then $\{ \psi_1, \ldots , \psi_n\}$ forms another coframe, and 
$$ d\varphi_{\alpha} = \lambda_1 \sqrt{-1}\psi_1\wedge \overline{\psi}_1 + \cdots +  \lambda_k \sqrt{-1}\psi_k\wedge \overline{\psi}_k, $$
Hence 
$$ (d\varphi_{\alpha})^k = \lambda_1 \cdots \lambda_k k! \Phi, \ \ \ \ \ \Phi = \sqrt{-1}\psi_1\wedge \overline{\psi}_1 \wedge \cdots \wedge \sqrt{-1}\psi_k\wedge \overline{\psi}_k. $$
This gives us a non-negative $(k,k)$-form $\Phi$ which is $d$-exact and non-trivial, contradicting with the conclusion of Lemma \ref{lemmaCao}. So we know that we must have $d\varphi_{\alpha}= 0$, or equivalently, $D^{\ast}_{\alpha \ast}=0$ for each $\alpha$. This completes the proof of the lemma. 
\end{proof}

Next let us organize all the possibly non-trivial structure constants into matrix form, to simplify our later discussions. Let us denote by
\begin{equation}
C_{\alpha}=(C^j_{i\alpha}), \ \ C_{a}=(C^j_{ia}), \ \ D_{\alpha}=(D^j_{i\alpha}), \ \ D_{a}=(D^j_{ia}), \ \ Z_a=(D^a_{ij})
\end{equation}
the $r\times r$ matrices with $(ij)$-th entry given above, and write
\begin{equation}
v^{y}_x=(D^{y}_{ix}), \ \ w_{xy}=(C^{i}_{xy}), \ \ \ \ \ r\!+\!1\leq x, \,y\leq n, 
\end{equation}
for the column vectors whose $i$-th entry is given above, where $x$ and $y$ can be either $\alpha$ or $a$. By Lemma \ref{restriction1} we have
\begin{equation}
C_{\alpha} =- D_{\alpha}^{\ast}, \ \ \ \ \ w_{\alpha \beta} = \overline{ v^{\beta}_{\alpha} } - \overline{ v^{\alpha}_{\beta} } . 
\end{equation}

\begin{lemma} \label{restriction3}
Let $(M^n,g)$ be a compact Hermitian-symplectic manifold, where $M=G/\Gamma$ is the quotient of a $2$-step solvable Lie group by a discrete subgroup and the complex structure $J$ is left-invariant. Let  $({\mathfrak g}, J,g)$ be the corresponding Hermitian Lie algebra. Then under any admissible frame $e$, the structure constants satisfy
\begin{eqnarray}
&& [C_x, C_y] \,= \,[D_x, D_y] \,=\,0,    \label{eq:C1}\\
&& [C_x^{\ast},D_y ] \, = \, -Z_x \overline{Z}_y, \label{eq:C2} \\
&& D_xZ_a-Z_a\,^t\!C_x \, = \, 0, \label{eq:C3} \\
&& C_x^{\ast}Z_y -Z_y \overline{D}_x \, =\,  C_y^{\ast}Z_x -Z_x \overline{D}_y,      \label{eq:C4}   \\
&& ^t\!C_x w_{yz} + \,^t\!C_y w_{zx} + \,^t\!C_z w_{xy}  \, = \,0 , \label{eq:C5} \\
&& D_x v^y_z - D_z v^y_x  + Z_y w_{xz} \, = \, 0, \label{eq:C6} \\
&& C_x^{\ast}v^z_y - C_z^{\ast}v^x_y + D_y \overline{w}_{xz} + Z_x \overline{v^y_z} - Z_z \overline{v^y_x }\, = \, 0. \label{eq:C7}
\end{eqnarray}
for any $r\!+\!1\leq x, y, z\leq n$.
\end{lemma}

\begin{proof}
Note that $Z_{\alpha}=0$ by Lemma \ref{restriction1}. Using the Bianchi identity (\ref{Bianchi}), which is equivalent to $d^2\varphi_t=0$, we can derive the above equations though a  straight-forward computation. We give the details here for readers' convenience. Since $C^{\alpha}_{\ast\ast} = C^a_{\ast\ast}=0$, the first identity in (\ref{Bianchi}) becomes
$$ \sum_{j=1}^r \big( C^i_{jx}C^j_{yz} + C^i_{jy}C^j_{zx} + C^i_{jz}C^j_{xy} \big) = 0, \ \ \ \ \ \forall \ 1\leq x,y,z\leq n. $$ 
If we let $x,y,z$ be all bigger than $r$, then we get (\ref{eq:C5}). If $x,y>r$ but $z=k\leq r$, then we get
$$ \sum_j (- C^i_{jx}C^j_{ky} + C^i_{jy}C^j_{kx}) =0, $$
that is, $C_xC_y-C_yC_x=0$, which is the first half of (\ref{eq:C1}). Similarly, since $D^{\ast}_{\alpha \ast}=D^{\ast}_{a \ast}=0$, the second identity of (\ref{Bianchi}) becomes
$$ \sum_{j=1}^r \big( C^j_{xy}D^z_{ij} + D^j_{ix}D^z_{jy} -  D^j_{iy}D^z_{jx} \big) =0, \ \ \ \ \forall \ 1\leq i\leq r, \ \ \ \forall \ 1\leq x,y,z,\leq n. $$
If both $x$ and $y$ are $\leq r$, then all terms are zero. If $x>r$ but $y=k\leq r$, then we have
$$ \sum_j (- C^j_{kx}D^z_{ij} + D^j_{ix} D^z_{jk} ) =0. $$
It is non-trivial only if $z=a>s$, in which case we get $Z_a\,^t\!C_x = D_xZ_a$, which is (\ref{eq:C3}). When both $x$ and $y$ are $>r$, if $z=k\leq r$, then we get $[D_x, D_y]=0$, which is the second half of (\ref{eq:C1}), and if $z>r$, then we have
$ D_xv^z_y - D_y v^z_x + Z_z w_{xy}=0$, which is (\ref{eq:C6}). Finally, let us examine the third equation in (\ref{Bianchi}), which becomes
$$ \sum_{j=1}^r \big( C^j_{xz}\overline{D^j_{iy} }  - C^i_{jz}\overline{D^x_{jy} }  + C^i_{jx}\overline{D^z_{jy} } - D^y_{jx}\overline{D^z_{ij} } + D^y_{jz}\overline{D^x_{ij} }  \big) =0, $$
where $i\leq r$ and $x,y,z$ are arbitrary. It is trivial when both $x$ and $z$ are $\leq r$. Let $x>r$ and $z=k\leq r$. Then we get
$$ \sum_{j=1}^r \big( C^j_{xk}\overline{D^j_{iy} }    + C^i_{jx}\overline{D^k_{jy} }  + D^y_{jk}\overline{D^x_{ij} }  \big) =0. $$
It is trivial when $y\leq r$, while when $y>r$ it is $-\overline{D}_y\,^t\!C_x + \,^t\!C_x \overline{D}_y + \overline{Z}_xZ_y =0$. After taking complex conjugation, it is just (\ref{eq:C2}).  Now suppose both $x$ and $z$ are $>r$. If $y=k\leq r$, then we get
$$ \sum_{j=1}^r \big( 0  - C^i_{jz}\overline{D^x_{jk} }  + C^i_{jx}\overline{D^z_{jk} } - D^k_{jx}\overline{D^z_{ij} } + D^k_{jz}\overline{D^x_{ij} }  \big) =0, $$
that is, $-^t\!C_z\overline{Z}_x + \,^t\!C_x\overline{Z}_z - \overline{Z}_z D_x + \overline{Z}_x D_z = 0$, which after conjugation is simply (\ref{eq:C4}). If $y>r$, on the other hand, then we get 
$$ \overline{D}_y w_{xz} -  \,^t\!C_z\overline{v^x_y} + \,^t\!C_x\overline{v^z_y} - \overline{Z}_z v^y_x + \overline{Z}_x v^y_z = 0, $$
which after conjugation is just (\ref{eq:C7}).  This completes the proof of the lemma.
\end{proof}

Next we will use Lemma \ref{lemmaGuo} to derive further restrictions on the structure constants $C$ and $D$ when the metric is assumed to be Hermitian-symplectic. We have the following:

\begin{lemma} \label{restriction4}
Let $(M^n,g)$ be a compact Hermitian-symplectic manifold, where $M=G/\Gamma$ is the quotient of a $2$-step solvable Lie group by a discrete subgroup and the complex structure $J$ is left-invariant. Let  $({\mathfrak g}, J,g)$ be the corresponding Hermitian Lie algebra. Then under any admissible frame $e$, there exists a skew-symmetric matrix $S$ so that the following hold:
\begin{eqnarray}
&& \langle u_x,\overline{v^z_y} \rangle \,= \,\langle u_z,\overline{v^x_y} \rangle   \label{eq:D1}\\
&& S' \overline{v^x_y} + D_y^{\ast} u_x  = \, \frac{\sqrt{-1}}{2} v^y_x , \label{eq:D2} \\
&& \overline{Z}_x u_y -  \overline{Z}_y u_x    \, = \, \frac{\sqrt{-1}}{2} w_{xy}, \label{eq:D3} \\
&& C_x +D_x  \, = \, -2\sqrt{-1} S'\overline{Z}_x, \label{eq:D4} \\
&& \,^t\!Z_x - Z_x \, = \, 2\sqrt{-1} (D^{\ast}_x S' + S' \overline{D}_x) , \label{eq:D5} \\
&& S'\,^t\!C_x + C_x S' \, = \, 0 , \label{eq:D6} \\
&& C_xu_y - C_yu_x - S' w_{xy} \, = \, 0. \label{eq:D7} \\
&& \langle u_x, w_{yz} \rangle   + \langle u_y, w_{zx} \rangle  + \langle u_z, w_{xy} \rangle  \, = \, 0. \label{eq:D8}
\end{eqnarray}
for any $r\!+\!1\leq x, y, z\leq n$. Here we have written $S'=(S_{ij})$ as a $r\times r$ (skew-symmetric) matrix and $u_x=(S_{ix})$ as a column vector in ${\mathbb C}^r$. 
\end{lemma}

\begin{proof}
By Lemma \ref{lemmaGuo}, we know that there will be a skew-symmetric matrix $S$ so that the two identities hold. Let us start with the second identity
\begin{equation} \label{HS2}
 \sum_{j=1}^r \big( S_{jx} \overline{D^z_{jy}} - S_{jz} \overline{D^x_{jy}} \big) = \frac{\sqrt{-1}}{2} \sum_{t=1}^n T^t_{xz} g_{t\bar{y}} =  \frac{\sqrt{-1}}{2} \sum_{j=1}^r \big( - C^j_{xz}g_{j\bar{y}} - D^y_{jz} g_{x\bar{j}} +  D^y_{jx} g_{z\bar{j}} \big) ,
 \end{equation}
for any $1\leq x,y,z\leq n$. Here we have used the formula in Lemma \ref{lemma1} which gives the Chern torsion components by structure constants under a frame $e$ which is not necessarily unitary. Note that for our admissible frame $e$ here, the matrix of metric $(g_{x\bar{y}})$ is block-diagonal in the three blocks corresponding to the decomposition of ${\mathfrak g}$ into three summands. The first and third block of $g$ are both the identity matrix, but the middle block $(g_{\alpha\bar{\beta}})$ might not be the identity matrix, as we chose $e$ so that $\{ e_{\alpha}+\overline{e}_{\alpha}\, \mid \, r<\alpha \leq s\}$ spans $V$, and $V$ might not be perpendicular to $JV$.  So at most two elements in $\{ x,y,z\}$  can be $\leq r$, otherwise the equation holds trivially. We divide the discussion into the following cases:

\vspace{0.1cm}
 
{\em Case 1.}   If  $x$, $y$, $z$ are all $>r$. 

\vspace{0.1cm}

In this case  (\ref{HS2}) becomes $
 \sum_{j} \big( S_{jx} \overline{D^z_{jy}} - S_{jz} \overline{D^x_{jy}} \big) =  0$, or equivalently $\langle u_x, \overline{v^z_y} \rangle = \langle u_z, \overline{v^x_y} \rangle $, which is (\ref{eq:D1}).

\vspace{0.1cm}
 
{\em Case 2.}   If exactly two of  $x$, $y$, $z$ are $>r$. 

\vspace{0.1cm}

If $x,y> r$ and $z=k\leq r$, then (\ref{HS2}) becomes $
 \sum_{j} \big( S_{jx} \overline{D^k_{jy}} - S_{jk} \overline{D^x_{jy}} \big)  =  \frac{\sqrt{-1}}{2}  D^y_{kx} 
$, which is (\ref{eq:D2}). If $x,z> r$, $y=k\leq r$, then (\ref{HS2}) becomes  $
 \sum_{j} \big( S_{jx} \overline{D^z_{jk}} - S_{jz} \overline{D^x_{jk}} \big)  =  - \frac{\sqrt{-1}}{2}   C^k_{xz}$, or (\ref{eq:D3}). 
 
 \vspace{0.1cm}
 
{\em Case 3.}   Only one of  $x$, $y$, $z$ is $>r$. 

\vspace{0.1cm}
 
First assume that $x>r$, $y=i\leq r$, $z=k\leq r$. Then (\ref{HS2}) becomes
$$ \sum_j \big( 0 - S_{jk} \overline{D^x_{ji}} \big) =   \frac{\sqrt{-1}}{2} \big( - C^i_{xk} + D^i_{kx} \big) , $$
That is, $S'\overline{Z}_x = \frac{\sqrt{-1}}{2} (C_x+D_x)$, so we get (\ref{eq:D4}).  If $y>r$, $x=i\leq r$ and $z=k\leq r$, then (\ref{HS2}) turns into
$
 \sum_j \big( S_{ji} \overline{D^k_{jy}} - S_{jk} \overline{D^i_{jy}} \big) =  \frac{\sqrt{-1}}{2}  (  - D^y_{ik}  +  D^y_{ki}  ) $, or $-S'\overline{D}_y -D^{\ast}_y S' = \frac{\sqrt{-1}}{2}  (\,^t\!Z_y-Z_y)$, which is (\ref{eq:D5}). 
 
If $x,y,z$ are all $\leq r$, then the equation (\ref{HS2}) is trivial, so we have exhausted all the possibilities for (\ref{HS2}). Next let us examine the first equation in Lemma \ref{lemmaGuo}, which in our case becomes
\begin{equation} \label{HS1}
\sum_{j=1}^r \big( S_{jx}C^j_{yz} + S_{jy} C^j_{zx}  + S_{jz} C^j_{xy} \big) =0, \ \ \ \ \ \forall \ 1\leq x,y,z\leq n.
\end{equation}

If all $x,y,z \leq r$, then the above equation is trivial. If $x>r$, $y=i\leq r$ and $z=k\leq r$, then the above equation reads $\sum_j (S_{ji}C^j_{kx} + S_{jk}C^j_{xi}) =0$, which is (\ref{eq:D6}). If $x,y>r$ and $z=i\leq r$, then (\ref{HS1}) becomes $\sum_j (S_{jx}C^j_{yi} + S_{jy}C^j_{ix} + S_{ji}C^j_{xy})=0$, which is (\ref{eq:D7}). Finally, if all $x,y,z$ are $>r$, then we get (\ref{eq:D8}). This completes the proof of the lemma. 
\end{proof}

From these matrix equations, we want to deduce the characterizations for Hermitian-symplectic metrics on $2$-step solvable Lie algebras, so we can eventually confirm Streets-Tian Conjecture in this case.

\vspace{0.2cm}

{\em Claim 1.} $Z_a=0$ for all $s\!+\!1\leq a\leq n$.


\begin{proof}
We already have $Z_{\alpha}=0$ by Lemma \ref{restriction1}. To show that $Z_a=0$, let us fix $a$, and multiply $\overline{Z}_a$ from the right onto (\ref{eq:D5}) for $x=a$, and skip the common subscript $a$, we get
\begin{eqnarray*}
 ^t\!Z\overline{Z} - Z\overline{Z}  & = &  2i(D^{\ast}S'+S'\overline{D})\overline{Z} \ \, = \ \, 2iD^{\ast}S'\overline{Z} +  2iS'\overline{(DZ)}    \\
 & = &  2iD^{\ast}S'\overline{Z} +  2iS' \overline{Z\,^t\!C} \ \, = \ \, D^{\ast}(2iS'\overline{Z}) + (2iS'\overline{Z}) C^{\ast} \\
 & = & - D^{\ast} (C+D) - (C+D)C^{\ast} .
 \end{eqnarray*}
Here at the second line we used (\ref{eq:C3}) and at the third line we used (\ref{eq:D4}). By (\ref{eq:C2}) for $x=y=a$, we have $C^{\ast}D - D C^{\ast} = - Z \overline{Z}$. Plug it into the above equation we obtain
$$ ^t\!Z\overline{Z} + C^{\ast}D - D C^{\ast}  = - D^{\ast} (C+D) - (C+D)C^{\ast} , $$
 or equivalently, 
$$  ^t\!Z\overline{Z} + (C+D)^{\ast} (C+D) = [ C^{\ast}, C]. $$
Taking trace on both sides, we conclude that $Z=0$ and $C+D=0$. This proves the claim.
 \end{proof}

\vspace{0.1cm}

{\em Claim 2.} $w_{xy}=0$ for all $r\!+\!1 \leq x,y \leq n$.


\begin{proof}
Since we already proved that $Z=0$, by (\ref{eq:D3}) we get $w_{xy}=0$ for all $x$, $y$.
\end{proof}

In summary, we know that for any $r\!+\!1 \leq x \leq n$, $C_x=-D_x$ and all $D_x$ are normal and commute with each other. So by a unitary change of $\{ e_1, \ldots , e_r\}$ we may assume that all $D_x$ are diagonal, which we will assume from now on. We already know that all $w_{xy}=0$, while $D$ and $v^y_x$ satisfy the following conditions from Lemma \ref{restriction3} and Lemma \ref{restriction4}:
\begin{equation} \label{summary}
\left\{ \begin{split} 
S' \overline{ v^x_y} + D^{\ast}_y u_x = \frac{i}{2} v^y_x,   \hspace{2.95cm}  \\
D_xS' + S'\,^t\!D_x =0, \ \ \ D_x^{\ast}S' + S'\overline{D}_x=0, \\
D_x v^y_z = D_z v^y_x , \ \ \ D^{\ast}_x v^z_y = D_z^{\ast} v^x_y , \hspace{1.27cm}\\
D_x u_y = D_y u_x, \ \ \ \langle u_x, \overline{ v^z_y} \rangle = \langle u_z, \overline{ v^x_y} \rangle.   \hspace{0.55cm} \\
D_{\alpha}^{\ast}=D_{\alpha}, \ \ \ v^{\alpha}_{\beta} = v^{\beta}_{\alpha},  \hspace{2.82cm}
\end{split} \right. 
\end{equation}
for any $r\!+\!1 \leq x,y,z \leq n$ and any $r\!+\!1 \leq \alpha, \beta \leq s$. Note that the last line of (\ref{summary}) was due to the fact that $C^i_{j\alpha}=-\overline{ D^j_{i\alpha} } $ and $C^j_{\alpha\beta} = \overline{ D^{\beta}_{j\alpha} } - \overline{ D^{\alpha}_{j\beta} }$.

\vspace{0.1cm}

{\em Claim 3.} For any $r\!+\!1 \leq x \leq n$, both $D_x$ and $D_x^{\ast}$ commute with $S'^{\ast}S'$. 


\begin{proof}
Since all our $D_x$ are diagonal now, the second line of (\ref{summary}) reads
$$ DS'=- S'D, \ \ \ \overline{D}S' = - S'\overline{D}. $$
Since $S'$ is skew-symmetric, we have
$$ D S'^{\ast}S'= - D  \overline{S'} S' = (- \overline{ \overline{D} S'}) S' = (\overline{S'\overline{D}}) S' = \overline{S'}(DS')=   \overline{S'}(-S' D) =  S'^{\ast}S'D. $$
Similarly, $D^{\ast} S'^{\ast}S' =  S'^{\ast}S'D^{\ast}$. This proves the claim.
\end{proof}

Let us denote by $I_x$ the image space of $D_x$. The next claim says that each $v^x_y$ lives in the intersection $I_x\cap I_y$.

\vspace{0.1cm}

{\em Claim 4.} For any $r\!+\!1 \leq x,y \leq n$, $v^x_y$ belongs to $I_x\cap I_y$. 


\begin{proof}
Since all $D_x$ are diagonal, the equation $D_xu_y=D_yu_x$ implies that $D_xu_y \in I_y$. Hence $D_x\overline{u}_y\in I_y$ as well. Let us write $D_x\overline{u}_y=D_y\xi$ for some column vector $\xi \in {\mathbb C}^r$. Multiplying $S'^{\ast}$ from the left side onto the first equation of (\ref{summary}), we get
\begin{eqnarray*}
S'^{\ast}S' \overline{v^x_y} & = & S'^{\ast} \big( \frac{i}{2} v^y_x - D^{\ast}_y u_x \big) \ \, = \ \, - \frac{i}{2} \,\overline{ ( S' \overline{v^y_x })}  - S'^{\ast} D^{\ast}_y u_x   \\
& = &  - \frac{i}{2} \,\overline{ (\frac{i}{2} v^x_y - D^{\ast}_x u_y   ) }  + D^{\ast}_y S'^{\ast}  u_x  \\
& = & - \frac{1}{4} \,\overline{  v^x_y} + \frac{i}{2} D_x \overline{u}_y + D^{\ast}_y S'^{\ast}  u_x \\
& = & - \frac{1}{4} \,\overline{  v^x_y} + \frac{i}{2} D_y \xi + D^{\ast}_y S'^{\ast}  u_x.
\end{eqnarray*}
Therefore, 
$$ (\frac{1}{4}I+ S'^{\ast}S' ) \overline{v^x_y} =  \frac{i}{2} D_y \xi + D^{\ast}_y S'^{\ast}  u_x. $$
By Claim 3, the positive definite matrix $(\frac{1}{4}I+ S'^{\ast}S' )$, hence its inverse, commute with $D_y$ and $D_y^{\ast}$, thus
$$ \overline{v^x_y} = D_y \{ \frac{i}{2} (\frac{1}{4}I+ S'^{\ast}S' )^{-1}\xi \} + D_y^{\ast} \{ (\frac{1}{4}I+ S'^{\ast}S' )^{-1}S'^{\ast}  u_x\}. $$
Thus $v^x_y\in I_y$. Similarly, $v^x_y\in I_x$, so the claim is proved. 
\end{proof}

\vspace{0.1cm}

{\em Claim 5.} There exist $\xi_x \in {\mathbb C}^r$ such that $v^x_y = D_y \xi_x$ for all $r\!+\!1 \leq x,y \leq n$. 


\begin{proof} 
Consider the function $\sigma : {\mathbb C}\rightarrow {\mathbb C}$ defined by $\sigma (0)=0$ and $\sigma (z)=z^{-1}$ for $z\neq 0$. Write 
$$ D_x = \mbox{diag} \{ \lambda_{x1}, \ldots , \lambda_{xr}\}, \ \ \ \ \  D_x^{\sigma} = \mbox{diag} \{ \sigma(\lambda_{x1}), \ldots , \sigma(\lambda_{xr}) \}. $$
Then clearly we have $D^{\sigma}_x D_x w = w$ for any vector $w\in I_x \subseteq {\mathbb C}^r$. For any $r\!+\!1 \leq x \leq n$, write $ \xi_x = D^{\sigma}_x (v^x_x)$. Then by the third line of (\ref{summary}) we have
$$ D_y \xi_x = D_y D^{\sigma}_x (v^x_x) = D^{\sigma}_x D_y v^x_x =  D^{\sigma}_x D_x v^x_y = v^x_y, $$
where the last equality holds as $v^x_y \in I_x$ by Claim 4. This proves the claim.
\end{proof}

Now we are ready to prove the main result of this article, Theorem \ref{thm}:

\begin{proof}[{\bf Proof of Theorem \ref{thm}.}] 
Let $(M^n,g)$ be a compact Hermitian-symplectic manifold such that $M=G/\Gamma$ is the quotient of a $2$-step solvable Lie group by a discrete subgroup, and the complex structure $J$ of $M$, when lifted onto $G$, is left-invariant. By the averaging lemma of Fino and Grantcharov \cite{FG}  (see \cite[Lemma 3.2]{EFV}), we know that the complex manifold $M$ admits a Hermitian-symplectic metric which when lifted to $G$  is left-invariant. We will still use $g$ to denote this left-invariant Hermitian-symplectic metric. This gives a Hermitian structure on the Lie algebra of $G$, denoted as $({\mathfrak g}, J, g)$, where ${\mathfrak g}$ is $2$-step solvable and $g$ is Hermitian-symplectic. Choose an admissible frame $e$, then the structure constants $C$ and $D$ satisfies all the restrictions given by Lemma \ref{restriction1}, Lemma \ref{restriction2}, Lemma \ref{restriction3}, and Lemma \ref{restriction4}, as well as Claims 1 through 5. In short, the only possibly non-trivial components of $C$ and $D$ are $D_x=-C_x$ which are all normal and commuting to each other, and $v^y_x$, for $r\!+\!1\leq x,y\leq n$, obeying all the conditions in (\ref{summary}). As before, we will modify our unitary basis $\{ e_1, \ldots , e_r\}$ for ${\mathfrak g}'_J$ so that all $D_x$ are diagonal. Write
$$ D^j_{ix} = \lambda_{xi}\delta_{ij}, \ \ \ \xi_x =\,^t\!(\xi_{x1}, \ldots , \xi_{xr}), \ \ \ \ \ 1\leq i,j\leq r, \ \ r\!+\!1 \leq x\leq n. $$
Note that for each $i$, the vector $t_i=(\lambda_{xi})_{x=r\!+\!1}^n$ must be non-zero (for the same reason, these vectors $\{ t_1, \ldots , t_r\}$ must be linearly independent), as otherwise we would have 
$$D^i_{\ast \ast} = D_{i\ast}^{\ast} = D^{\ast}_{\ast i}=0, \ \ \ C^i_{\ast\ast} = C^{\ast}_{i\ast}=0, $$
which would contradict with the fact that $e_i+\overline{e}_i \in {\mathfrak g}'$. By the second equation of the third line of (\ref{summary}), for any $r\!+\!1\leq x,y,z\leq n$, we have $D_x^{\ast}v^z_y = D_z^{\ast}v^x_y$, or equivalently, $D_x^{\ast}D_y\xi_z = D_z^{\ast}D_y\xi_x$, that is
$$ \overline{\lambda}_{xi} \lambda_{yi} \xi_{zi} =  \overline{\lambda}_{zi} \lambda_{yi} \xi_{xi}, \ \ \ \ \  \forall \ 1\leq i\leq r. $$
Fix any $i$, and let $y$ exhaust all indices from $r+1$ to $n$, since $t_i\neq 0$, we must have
$$ \overline{\lambda}_{xi}  \xi_{zi} =  \overline{\lambda}_{zi} \xi_{xi}, \ \ \ \ \  \forall \ 1\leq i\leq r, \ \ \forall 
\ r\!+\!1 \leq x,z\leq n. $$
Thus the vector $\tau_i = (\xi_{xi})_{x=r+1}^n$ must be proportional to $\overline{t}_i$. Let us write 
$ \xi_{xi} = \overline{p}_i \overline{\lambda}_{xi}$ for some constant $p_i$, then we have
\begin{equation*}
 D^x_{iy} = (v^x_y)_i = \lambda_{yi}\xi_{xi} = \overline{p}_i  \lambda_{yi} \overline{ \lambda}_{xi} .
 \end{equation*}
From the structure equation (\ref{eq:structure}) and (\ref{summary}), we get $d\varphi_x=0$ for each $r+1\leq x\leq n$, and 
\begin{eqnarray}
d\varphi_i & = & -\sum_{x=r+1}^n  C^i_{ix}\varphi_i\varphi_x - \sum_{x=r+1}^n \overline{D^i_{ix}}\varphi_i\overline{\varphi}_x - \sum_{x,y=r+1}^n \overline{D^y_{ix}}\varphi_y\overline{\varphi}_x \nonumber \\
& = & \varphi_i \sum_{x=r+1}^n ( \lambda_{xi}\varphi_x-  \overline{\lambda}_{xi}\overline{\varphi}_x ) - \sum_{x,y=r+1}^n\overline{\lambda}_{xi} p_i\lambda_{yi} \varphi_y\overline{\varphi}_x \nonumber  \\
& = & \varphi_i (\sigma_i -\overline{\sigma}_i ) - p_i \sigma_i \overline{\sigma}_i,  \label{dvarphi}
\end{eqnarray}
where $\sigma_i =  \sum_{x=r+1}^n \lambda_{xi}\varphi_x$. Recall that the K\"ahler form $\omega$ of our metric $g$ is
\begin{equation} \label{eq:Kahler}
 \omega = \sqrt{-1}\sum_{i=1}^r \varphi_i \overline{\varphi}_i + \sqrt{-1}\!\sum_{\alpha ,\beta =r+1}^s \! g_{\alpha \bar{\beta}} \,\varphi_{\alpha} \overline{\varphi}_{\beta}  + \sqrt{-1}\!\sum_{a=s+1}^n \! \varphi_a \overline{\varphi}_a . 
 \end{equation}
The formula (\ref{dvarphi}) yields that
$$ d\omega = - \sqrt{-1}\sum_{i=1}^r (\overline{p}_i\varphi_i + p_i \overline{\varphi}_i) \sigma_i \overline{\sigma}_i, $$
which is not zero in general. In other words, the Hermitian-symplectic metric $g$ may not be K\"ahler in general.  Consider the following set of linearly independent $(1,0)$-forms:
\begin{equation*}
\psi_i = \varphi_i + p_i \sigma_i, \ \ \ 1\leq i\leq r; \ \ \ \mbox{and} \ \ \ \psi_x =\varphi_x,\ \ \ r+1\leq x\leq n.  
\end{equation*}
They give us a positive $(1,1)$-form
\begin{equation*}
\tilde{\omega} =\sqrt{-1}\sum_{i=1}^r \psi_i \overline{\psi}_i + \sqrt{-1}\sum_{\alpha ,\beta =r+1}^s g_{\alpha \bar{\beta}} \,\psi_{\alpha} \overline{\psi}_{\beta}  + \sqrt{-1}\sum_{a=s+1}^n \psi_a \overline{\psi}_a 
\end{equation*}
So $\tilde{\omega}$ becomes the K\"ahler form of a new Hermitian metric $\tilde{g}$ on $({\mathfrak g}, J)$. We have $d\psi_x=0$ for any $r+1\leq x\leq n$, while for any $1\leq i\leq r$, 
\begin{eqnarray*}
d(\psi_i \overline{\psi}_i) & = & d\psi_i \,\overline{\psi}_i - \psi_i \,\overline{d\psi_i} \ \, = \ \, d\varphi_i \,\overline{\psi}_i - \psi_i \,\overline{d\varphi_i} \\
& = & \big( \varphi_i (\sigma_i-\overline{\sigma}_i ) - p_i \sigma_i \overline{\sigma}_i \big) ( \overline{\varphi}_i  + \overline{p}_i \overline{\sigma}_i ) - (\varphi_i + p_i \sigma_i) \,  \big( -\overline{\varphi}_i (\sigma_i-\overline{\sigma}_i ) + \overline{p}_i \sigma_i \overline{\sigma}_i \big) \\
& = & 0.
\end{eqnarray*}
Therefore we have $d\tilde{\omega}=0$ and $\tilde{g}$ is K\"ahler. This shows that the presence of a Hermitian-symplectic metric on ${\mathfrak g}$ leads to the existence of K\"ahler metric, so Streets-Tian Conjecture holds for compact quotients of all $2$-step solvable groups. 
\end{proof}

Note that from the proof above, we get an explicit description of Hermitian-symplectic metrics on $({\mathfrak g},J)$, namely those where the structure equations can be given as
$$ d\varphi_i = \varphi_i (\sigma_i - \overline{\sigma}_i ) - p_i \sigma_i \overline{\sigma}_i,   \ \ (1\leq i\leq r); \ \ \ \ \ \ 
d\varphi_x = 0, \ \ (r+1\leq x\leq n). $$
where $\sigma_i = \sum_{x=r+1}^n \lambda_{xi}\varphi_x$, and the K\"ahler form is given by (\ref{eq:Kahler}). Also, the last line in (\ref{summary}) indicates that $\lambda_{\alpha i}$ is real for each $\alpha$ and each $i$, and when $J{\mathfrak g}'\neq {\mathfrak g}'$, namely when there are $e_{\alpha}$ terms, each $p_i$ is also real. One can write down explicit examples of $2$-step solvable Lie algebras equipped with Hermitian-symplectic metrics.

\vspace{0.3cm}

\vs

\noindent\textbf{Acknowledgments.} The second named author would like to thank Bo Yang and Quanting Zhao for their interests and/or helpful discussions.

\vs

\end{document}